\documentclass[a4,12pt]{amsart}
\usepackage{amssymb,amstext,amsmath,amscd,amsthm,amsfonts,enumerate,latexsym}
\usepackage{color}
\theoremstyle{plain}
\newtheorem{thm}{Theorem}[section]
\newtheorem{prop}[thm]{Proposition}

\newtheorem{lem}[thm]{Lemma}
\newtheorem{cor}[thm]{Corollary}

\newtheorem*{claim*}{Claim}

\theoremstyle{definition}
\newtheorem{defn}[thm]{Definition}

\newtheorem{ex}[thm]{Example}

\newtheorem{rem}[thm]{Remark}

\newtheorem{conj}[thm]{Conjecture}

\theoremstyle{remark}

\numberwithin{equation}{thm}

\def\Ker{\operatorname{Ker}}

\def\mod{\mathrm{mod}}

\newcommand{\rmI}{\mathrm{I}}

\newcommand{\fkm}{\mathfrak{m}}\newcommand{\fkn}{\mathfrak{n}}

\def\gr{\mbox{\rm gr}}

\def\PF{\operatorname{PF}}

\def\gcd{\operatorname{gcd}}

\title[PF numbers and defining ideals in stretched numerical semigroup rings]{Pseudo-Frobenius numbers and defining ideals in stretched numerical semigroup rings}

\author{Do Van Kien}
\address{Department of Mathematics, Hanoi Pedagogical University 2, Phuc Yen, Vinh Phuc, Vietnam}
\email{dovankien@hpu2.edu.vn}

\author{Naoyuki Matsuoka}
\address{Department of Mathematics, School of Science and Technology, Meiji University, 1-1-1 Higashi-mita, Tama-ku, Kawasaki 214-8571, Japan}
\email{naomatsu@meiji.ac.jp}

\author{Taiga Ozaki}
\address{Department of Mathematics, School of Science and Technology, Meiji University, 1-1-1 Higashi-mita, Tama-ku, Kawasaki 214-8571, Japan}
\email{ozaki0422math.meijiu@gmail.com}



\newcommand{\detid}{\operatorname{I}}
\newcommand{\Apery}{\operatorname{Ap}}

\newcommand{\CMtype}{\operatorname{r}}

\begin{document}
\maketitle

\setlength{\baselineskip}{17pt}


\begin{abstract}
	The pseudo-Frobenius numbers of a numerical semigroup $H$ are deeply connected to the structure of the defining ideal of its semigroup ring $k[H]$. 
	In this paper, we resolve a certain conjecture related to this connection under the assumption that $k[H]/(t^a)$ is stretched, where $a$ is the multiplicity of $H$. 
	Furthermore, we provide numerical conditions for the tangent cone of $k[H]$ to be Cohen-Macaulay.
\end{abstract}

\section{Introduction}

In commutative algebra, the study of defining ideals and free resolutions of finitely generated algebras over a field has been a fundamental area of research, as these structures contain rich information about the algebraic properties of the rings. Among the various examples studied in commutative algebra, numerical semigroup rings stand out as particularly intriguing objects due to their seemingly simple structure as one-dimensional finitely generated algebras.

Despite their simplicity, our understanding of the defining ideals of numerical semigroup rings remains surprisingly limited. 
For instance, Herzog \cite{herzog} fully characterized the defining ideals of 3-generated numerical semigroup rings, but for semigroups with more generators, the situation becomes significantly more complex. For 4-generated semigroups, partial results have been obtained: symmetric semigroups (\cite{bresinsky}), pseudo-symmetric semigroups (\cite{komeda}), and almost symmetric semigroups (\cite{eto}). 
For $n(\ge 4)$-generated cases, research has primarily focused on specific classes, such as those generated by arithmetic sequences \cite{gimenez} or repunit numbers \cite{branco, rosales1}, leaving a comprehensive understanding out of reach.

A deeper understanding of these defining ideals could have applications in areas such as coding theory and algebraic geometry, where numerical semigroups naturally arise.
In fact, the numerical semigroup ring $k[H]$ of a numerical semigroup $$H=\left<a_1, a_2,\ldots , a_n\right> =\{\lambda_1 a_1 + \lambda_2 a_2 + \cdots + \lambda_n a_n \mid 0 \le \lambda_1, \lambda_2, \ldots \lambda_n \in \mathbb{Z}\}$$ corresponds to the monomial curve $C(H) = \{(t^{a_1}, t^{a_2}, \ldots , t^{a_n}) \mid t \in k\} \subseteq k^n$, which geometrically represents a one-dimensional algebraic curve embedded in $k^n$. The defining ideal $I_H$ of $k[H]$ describes the defining equation of $C(H)$.

Here let us introduce the definition of pseudo-Frobenius numbers of a numerical semigroup, which plays an important role in the present paper.
For a numerical semigroup $H$, we put
$$
\PF(H) = \{\alpha \in \mathbb{Z} \setminus H \mid \alpha + h \in H\text{ for all $0< h \in H$}\}
$$
and we call an element $\alpha \in \PF(H)$ a pseudo-Frobenius number of $H$. The maximum integer in $\PF(H)$, which is trivially equal to $\max(\mathbb{Z} \setminus H)$ is called the Frobenius number of $H$.
In the first and second authors' previous collaborative work \cite{goto} with S. Goto and H. L. Truong, we discovered that the behavior of pseudo-Frobenius numbers of numerical semigroups could potentially describe the structure of defining ideals of their associated semigroup rings. This observation led us to the following conjecture.

\begin{conj}[by discussion between the the first author, the second author, D. T. Cuong, and H. L. Truong]\label{conj}
Let $H = \left< a_1, a_2,\ldots , a_n\right>$ be an $n$-generated numerical semigroup. Then the following conditions are equivalent.
\begin{enumerate}
    \item $I_H = \detid_2 \left( \begin{smallmatrix} f_1 & f_2 & \cdots & f_n\\ g_1 & g_2 & \cdots & g_n \end{smallmatrix} \right)$ for some homogeneous elements $f_1, f_2, \ldots , f_n, g_1, g_2, \ldots , g_n \in S = k[X_1,X_2,...,X_n]$ with positive degrees.
    \item $I_H = \detid_2 \left( \begin{smallmatrix} X_1^{\ell_1} & X_2^{\ell_2} & \cdots & X_{n-1}^{\ell_{n-1}} & X_n^{\ell_n}\\ X_2^{m_2} & X_3^{m_3} & \cdots & X_n^{m_n} & X_1^{m_1} \end{smallmatrix} \right)$ for some integers $\ell_1, \ell_2, \ldots, \ell_n, m_1, m_2, \ldots ,m_n > 0$, after suitable permutations on $a_1, a_2,\ldots , a_n$.
    \item The set $\PF(H)$ forms an arithmetic sequence of length $n-1$.
\end{enumerate}
Here, $\detid_2(M)$ denotes the ideal of $S$ generated by $2 \times 2$ minors of a matrix $M$ whose entries are in $S$.
\end{conj}

While this conjecture remains unproven in its full generality, we have verified it in several important cases, including almost symmetric semigroups \cite{goto} and semigroups with maximal embedding dimension \cite{kien}. Additionally, Branco, Cola\c{c}o, and Ojeda \cite{branco, colaco} showed that the assertions in the conjecture are always satisfied for generalized repunit numerical semigroups, although these results were obtained independently of our conjecture.

In this paper, we prove that our conjecture holds under the assumption that $k[H]/(t^{a_1})$ is a stretched Artinian local ring. This condition is automatically satisfied when $H$ has maximal embedding dimension, making our result a significant step toward the complete resolution of the conjecture. 

The structure of stretched Artinian local rings has been studied by Sally \cite{sally1} for the Gorenstein case, and later extended to the general case by Elias and Valla \cite{elias}, who provided a complete description of their defining ideals. 
However, understanding the defining ideal of $k[H]$ itself from that of $k[H]/(t^{a_1})$ presents significant challenges. 
Our main contribution is demonstrating how these challenges can be overcome when $\PF(H)$ forms an arithmetic sequence of length $n-1$.

To state our main theorem, let us fix some notation. Let $k$ be a field and $H = \left<a_1, a_2,\ldots , a_n\right>$ be a numerical semigroup, namely $\gcd(a_1,a_2,\ldots , a_n) = 1$.
We put $S=k[X_1, X_2,\ldots , X_n]$ the polynomial ring over $k$ with the grading $\deg X_i = a_i$, and consider the graded ring homomorphism $\varphi_H : S \to k[H]$ defined by $\varphi(X_i) = t^{a_1}$ for each $1 \le i \le n$. Let $I_H = \Ker \varphi_H$ be the defining ideal of $k[H]$. For a matrix $M$ whose entries are in $S$, $\detid_2(M)$ denotes the ideal of $S$ generated by $2 \times 2$ minors of $M$.

\begin{thm}\label{main theorem}
	Let $H = \left<a_1, a_2,\ldots , a_n\right>$ be an $n(\ge 3)$-generated numerical semigroup with multiplicity $a_1$, and suppose that $k[H]/(t^{a_1})$ is stretched. Then the following are equivalent.
	\begin{enumerate}
		\item $I_H = \detid_2 \begin{pmatrix} f_1 & f_2 & \cdots & f_n\\g_1 & g_2 & \cdots & g_n\end{pmatrix}$ for some homogeneous elements $f_i, g_i \in S$ with positive degrees.
		\item $I_H = \detid_2 \begin{pmatrix} X_1^{h_1} & X_2 & \cdots & X_{n-1} & X_n \\ X_2 & X_3 & \cdots & X_n^\ell & X_1^p\end{pmatrix}$ for some positive integers $h_1, \ell, p$, after suitable permutations on $a_2, a_3,\ldots, a_n$.
		\item $\PF(H) =\{h+\alpha, h+ 2\alpha, \ldots , h+(n-1)\alpha\}$ for some $h \ge 0$ and $\alpha > 0$.
	\end{enumerate}
	When this is the case, the graded minimal $S$-free resolution of $k[H]$ is given by the Eagon-Northcott complex \cite{eagon} of the matrix appearing in the conditions (1) and (2).
\end{thm}

This connection between numerical properties (pseudo-Frobenius numbers) and algebraic structures (defining ideals) is particularly interesting. 
It is well-known that $\PF(H)$ corresponds to the degrees of generators of the graded canonical module of $k[H]$ (see \cite{goto2}). This fact suggests that $\PF(H)$ contains information on the ring structure of $k[H]$. 
On the other hand, it is clear that the ring structure is more complicated than the numerical information in $\PF(H)$ can express. 
In fact, the implication from (1) to (3) in Theorem \ref{main theorem} can be proven for arbitrary numerical semigroup $H$, by considering the grading of the Eagon-Northcott complex (\cite{eagon}). 
See \cite[Section 2.1]{kien} for the details.
Hence, if it is proven that specific patterns in $\PF(H)$, such as forming an arithmetic sequence, can completely determine the structure of defining ideals, then it should be marked a significant step forward. 
Our result highlights how combinatorial invariants of numerical semigroups can directly inform the algebraic structure of their associated semigroup rings, assuming the stretchedness in $k[H]/(t^{a_1})$, providing a unified perspective that bridges these two areas.

The organization of this paper is as follows. In Section 2, we review basic definitions and properties of numerical semigroups, their pseudo-Frobenius numbers, and stretched Artinian local rings. Section 3 contains our main technical results, where we establish the connection between the arithmetic sequence property of PF(H) and the structure of defining ideals under the stretched condition. 
In Section 4, we also discuss the Cohen-Macaulay property of the associated graded ring $\gr(k[H]) = \bigoplus_{\ell \ge 0} \fkm^\ell / \fkm^{\ell+1}$ of the graded maximal ideal $\fkm$ in $k[H]$.
Section 5 consists of several examples illustrating our main theorem and discussing potential generalizations.


\section{Preliminaries}

	In this section, the notation is defined as follows.
	Let $H$ be a numerical semigroup minimally generated by $a_1, a_2,\ldots , a_n$.
	We always assume $a_1 = \min(H \setminus\{0\})$, which is called the multiplicity of $H$.
	Let $k$ be a field, $S=k[X_1, X_2,\ldots , X_n]$ and $V=k[t]$ be the polynomial rings over $k$.
	The numerical semigroup ring of $H$, denoted by $k[H]$, is defined as  
	$$
	k[H] = k[t^{a_1}, t^{a_2}, \ldots , t^{a_n}] \subseteq k[t].
	$$

\subsection{Fundamental notation in numerical semigroups}

	The Ap\'ery set and the pseudo-Frobenius numbers are essential tools in the study of numerical semigroups.
	First, let us recall the definition of them.
	
	\begin{defn}
		For $0 < h \in H$, we put
		$$
		\Apery(H, h) = \{a \in H \mid a-h \notin H\}.
		$$
		We put $w_i = \min \{a \in H \mid a \equiv i \pmod h\}$ for each $0 \le i \le h-1$. Then, obviously we have $\Apery(H,h) = \{w_0 = 0, w_1, \ldots ,w_{h-1}\}$.
	\end{defn}

	\begin{defn}
		An integer $\alpha \in \mathbb{Z} \setminus H$ is called a pseudo-Frobenius number of $H$, if $\alpha + h \in H$ for all $0 < h \in H$. The set of all pseudo-Frobenius numbers of $H$ is denoted by $\PF(H)$.
	\end{defn}

	The following two lemmas are fundamental and used frequently.

	\begin{lem}\cite[Proposition 2.19 (2)]{rosales1}\label{PFlem1}
		For an integer $z \in \mathbb{Z}$, $z \notin H$ if and only if there exists $\alpha \in \PF(H)$ such that $\alpha - z \in H$.
	\end{lem}

	\begin{lem}\cite[Proposition 2.20]{rosales1}\label{PFlem2}
		For each $0 < h \in H$, we have
		$$
		\PF(H) = \{a - h \mid \text{$a \in \Apery(H,h)$ and $b-a \notin H$ for all $b \in \Apery(H,h) \setminus \{a\}$}\}
		$$
	\end{lem}

	Then, we have the following. Although the proof can be found in \cite{kien}, let us state a brief proof for the sake of completeness.

	\begin{lem}\cite[Lemma 2]{kien}
		Suppose that $n \ge 3$.
    	If there are integers $h \ge 0$ and $\alpha > 0$ such that $\PF(H) = \{h+\alpha, h+2\alpha, \ldots , h +(n-1)\alpha\}$, then $h \in H$ and $\alpha \notin H$.
	\end{lem}

	\begin{proof}
		First, suppose that $\alpha \in H$. Then $(h+\alpha) + \alpha \in H$, because $\alpha > 0$ and $h + \alpha \in \PF(H)$. However, since $n \ge 3$, $h + 2\alpha \notin H$, which is a contradiction. Hence $\alpha \notin H$. 
		Then, thanks to Lemma \ref{PFlem1}, there exists $1 \le k \le n-1$ such that $(h + k\alpha) - \alpha \in H$. This implies that $k = 1$. Hence $h \in H$.
	\end{proof}

\subsection{Stretched local rings and stretched numerical semigroup rings}

	The notion of stretched local rings was studied by several researchers, such as J. Sally, J. Elias, and G. Valla \cite{elias,sally1}. 
    We begin by recalling the definition of stretched Artinian local rings.
    
    \begin{defn}[\cite{sally1}]
    	Let $A$ be an Artinian local ring with maximal ideal $\fkm$.
    	We say that $A$ is stretched, if $\fkm^2 = (0)$ or $\fkm^2$ is principal.
    \end{defn}

    Let us review their structure theorem for stretched Artinian local rings, which is not directly used in this paper.

	\begin{thm}[{\cite[Corollary 1.2]{sally1} for Gorenstein case, \cite[Theorem 3.1]{elias} for general case}]
		Let $R$ be an $n$-dimensional regular local ring with maximal ideal $\fkn$. 
		Let $I$ be an $\fkn$-primary ideal of $R$ satisfying $I \subseteq \fkn^2$. 
		Put $A=R/I$ and $\fkm = \fkn/I$.
		Let $\ell$ denote the socle degree of $A$, that is, $\ell$ is the largest integer such that $\fkm^\ell \ne (0)$.
		Suppose that $A$ is stretched. Then we have the following.
		\begin{enumerate}
			\item If the Cohen-Macaulay type $\CMtype(A)$ is less than $n = \dim R$, then there exists a system $X_1, X_2, \ldots , X_n$ of generators of $\fkn$ such that $I$ is minimally generated by 
			$$
			\{X_iX_j \mid 1 \le i < j \le n\} \cup \{X_j^2\mid 2\le j \le \CMtype(A)\} \cup \{X_i^2 - u_i X_1^\ell\ \mid \CMtype(A) + 1 \le i \le n\}
			$$
			where $u_i$ are units in $R$.
			\item If $\CMtype(A) = n$, then there exists a system $X_1, X_2,\ldots , X_n$ of generators of $\fkn$ such that $I$ is minimally generated by
			$$
				\{X_1 X_j \mid 2\le j \le n\} \cup \{X_i X_j\mid 2\le i \le j \le n\} \cup \{X_1^{\ell+1}\}.
			$$
		\end{enumerate}
	\end{thm}

	We consider the condition that $k[H]/(t^{a_1})$ is a stretched Artinian local ring. Here, notice that, although $k[H]$ is not a local ring, $k[H]/(t^{a_1})$ is a local ring.
	
	\begin{rem}
		In \cite{sally1}, Sally defined the stretched property for arbitrary Cohen-Macaulay local rings as follows:
		\begin{quote}
			A Cohen-Macaulay local ring $A$ with maximal ideal $\fkm$ is said to be a stretched local ring, if there exists a parameter ideal $Q$ for $A$ such that $Q$ is a reduction of $\fkm$ and $A/Q$ is a stretched Artinian local ring.
		\end{quote}
		Henceforth, we can consider the condition that the local ring $k[[H]]$ is stretched, this is equivalent to saying that there exists $f = t^{a_1} + \text{(higher terms)}$ such that $k[[H]]/(f)$ is a stretched Artinian local ring. However, this is not equivalent to the condition that $k[[H]]/(t^{a_1}) (\cong k[H]/(t^{a_1}))$ is stretched.
		For example, $H=\left<6,7,11,15\right>$ satisfies that $k[[H]]$ is stretched but $k[[H]]/(t^6)$ is not stretched \cite{eto2}.

		Because the case where $k[[H]]$ is stretched but $k[[H]]/(t^{a_1})$ is not stretched is much more complicated (cf. \cite{eto2}), we concentrate on the case where $k[[H]]/(t^{a_1}) \cong k[H]/(t^{a_1})$ is stretched in this paper.
	\end{rem}

	The following lemma is the numerical criterion for stretchedness of $k[H]/(t^{a_1})$ in terms of the Ap\'{e}ry set $\Apery(H,a_1)$.


	\begin{lem}\label{Apery}
		Let $H = \left<a_1, a_2,\ldots , a_n\right>$ be a numerical semigroup with embedding dimension $n$.
		We put $\ell = a_1 - n + 1$.
		Then $k[H]/(t^{a_1})$ is stretched if and only if
		$$
		\Apery(H,a_1) = \begin{cases}
			\{0, a_2, a_3, \ldots , a_n\} & (\ell = 1)\\
			\{0, a_2, a_3, \ldots , a_n\} \cup \{a_\lambda + a_\mu\} & (\ell = 2)\\
			\{0, a_2, a_3, \ldots , a_n\} \cup \{2a_\lambda, 3a_\lambda, \ldots , \ell a_\lambda\} & (\ell \ge 3)
 		\end{cases}
		$$
		for some $1 \le \lambda \le \mu \le n$.
	\end{lem}
	
	\begin{proof}
		If $\ell =1$, i.e., $H$ has maximal embedding dimension, then the conclusion is clear. 
		
		Suppose that $\ell \ge 2$ and let $\fkm$ denote the maximal ideal of $k[H]/(t^{a_1})$. 
		Notice that 
		$$
		k[H]/(t^{a_1}) = \sum_{h \in \Apery(H, a_1)} k \overline{t^{h}}.
		$$
		Since $k[H]/(t^{a_1})$ is stretched, we also have
		$$
		k[H]/(t^{a_1}) \cong k + \fkm/\fkm^2 + \fkm^2/\fkm^3 + \cdots + \fkm^\ell/\fkm^{\ell+1}
		$$
		and $\dim_k \fkm^i/\fkm^{i+1} = 1$ for all $2\le i \le \ell$.
		
		Let $h \in \Apery(H, a_1)$ such that $0 \ne \overline{t^h} \in \fkm^2/\fkm^3$.
		Then we can write $h = a_\lambda + a_\mu$ for some $2 \le \lambda \le \mu \le n$.
		If $\ell = 2$, then $\fkm^3 = (0)$ and $$k[H]/(t^{a_1}) = k + \sum_{i=2}^n k\overline{t^{a_i}} + k\overline{t^{a_\lambda + a_\mu}}.$$
		Hence $\Apery(H,a) = \{0, a_2, \ldots, a_n\} \cup \{a_\lambda + a_\mu\}$.

		Suppose that $\ell \ge 3$. Then $\fkm^3/\fkm^4 \ne (0)$. 
		If $\lambda \ne \mu$, then $h + a_\rho - a_1 \in H$ for all $2\le \rho \le n$, which implies the contradiction that $\fkm^3/\fkm^4 = (0)$. Hence $\lambda = \mu$ and $\{\overline{t^{i a_\lambda}}\}$ is a basis of $k$-vector space $\fkm^i/\fkm^{i+1}$ for all $2 \le i \le \ell$.
		Thus, $\Apery(H,a_1) = \{0, a_2, \ldots ,a_n\} \cup \{2a_\lambda, 3a_\lambda, \ldots , \ell a_\lambda\}$ as desired.
	\end{proof}

\subsection{Defining ideals of numerical semigroup rings}

	The following is typically used in determining the defining ideal $I_H$ of $k[H]$.

	\begin{prop}\label{Nak}
		Let $J$ be a graded ideal of $S=k[X_1, X_2,\ldots , X_n]$ such that $J \subseteq I_H$.
		If $\dim_k S/(J + (X_1)) = a_1$, then $I_H = J$. 
	\end{prop}

	\begin{proof}
		Since $J + (X_1) \subseteq I_H + (X_1)$ and $\dim_k S/(J+(X_1) = a_1 = \dim_k k[H]/(t^{a_1}) = \dim_k S/(I_H +(X_1))$, we have $J+(X_1) = I_H +(X_1)$. By applying the graded Nakayama's lemma, we conclude $I_H = J$.
	\end{proof}

	In this paper, the following lemma provides a method to find such $J$ as in Proposition \ref{Nak}.
	Here, remember that $\deg X_i = a_i$ for each $1 \le i \le n$ in $S=k[X_1, X_2,\ldots , X_n]$.

	\begin{lem}
		Let $f_1,f_2, g_1,g_2\in S_+$ be monomials in $X_1, X_2,\ldots , X_n$. 
		Then $f_1 g_2 - f_2 g_1 \in I_H$ if and only if $\deg g_1 - \deg f_1 = \deg g_2 - \deg f_2$.
	\end{lem}

	\begin{proof}
		We put $p_i = \deg f_i$ and $q_i = \deg g_i$ for each $i = 1,2$.
		Then, $\varphi_H(f_1 g_2 - f_2 g_1) = t^{p_1 + q_2} - t^{p_2 + q_1}$.
		Hence $f_1 g_2 - f_2 g_1 \in \Ker \varphi_H = I_H$ if and only if $p_1 + q_2 = p_2 + q_1$, which is equivalent to $\deg g_1 - \deg f_1 = \deg g_2 - \deg f_2$.
	\end{proof}

	\begin{cor}\label{commondiff}
		Let $\ell_1, \ell_2,\ldots , \ell_n, m_1, m_2,\ldots , m_n$ be positive integers and put $M=\left(\begin{smallmatrix}X_1^{\ell_1} & X_2^{\ell_2} & \cdots & X_{n-1}^{\ell_{n-1}} & X_n^{\ell_n}\\ X_2^{m_2} & X_3^{m_3} & \cdots & X_n^{m_n} & X_1^{m_1} \end{smallmatrix}\right)$.
		If $$m_2 a_2 - \ell_1 a_1 = m_3 a_3 - \ell_2 a_2 = \cdots = m_n a_n - \ell_{n-1} a_{n-1} = m_1 a_1 - \ell_n a_n,$$ then $\detid_2 (M) \subseteq J$.
	\end{cor}

\section{Proof of Theorem \ref{main theorem}}

In this section, we prove Theorem \ref{main theorem} by demonstrating the following statement.

\begin{thm}\label{thmSec3}
	Let $H = \left<a_1, a_2, \ldots , a_n\right>$ be a numerical semigroup with embedding dimension $n\ge 3$ and multiplicity $a_1\ge 3$. 
	Suppose that $k[H]/(t^{a_1})$ is stretched and that $\PF(H) = \{h + \alpha, h+ 2\alpha , \ldots , h+ (n-1)\alpha\}$ for some $h \ge 0$ and $\alpha > 0$.
	Then, after a suitable permutation of $a_2, a_3, \ldots , a_n$, we have
	$$
	I_H = \begin{pmatrix}
		X_1^{h_1} & X_2^\ell & X_3 & \cdots & X_{n-1} & X_n\\
		X_2 & X_3 & X_4 & \cdots & X_n & X_1^p
 	\end{pmatrix}
	$$
	for some positive integers $h_1, \ell, p$.
\end{thm}

To prove Theorem \ref{thmSec3}, we first consider what happens when $k[H]/(t^{a_1})$ is stretched and $\PF(H) = \{h + \alpha, h+ 2\alpha , \ldots , h+ (n-1)\alpha\}$ for some $h \ge 0$ and $\alpha > 0$. 
Observe that we may assume $H$ does not have maximal embedding dimension.

Before starting the observation, we begin with the following crucial lemma for our proof.
Although the statement looks elementary, we provide the proof for the sake of completeness.

\begin{lem}\label{elementary}
	Let $\ell \ge 2$ and $n \ge 3$. Put $a = \ell + n -1$. For an integer $1 \le r \le n-2$, if $\{0, 1, \ldots , \ell\} \cup \{\ell + r, \ell + 2r, \ldots , \ell + (n-2)r\}$ is a complete system of representatives modulo $a$, then $r=1$.
\end{lem}

\begin{proof}
	When $n=3$, the statement is clear. Thus, we may assume $n \ge 4$.

	Let $S_1 = \{0, 1,\ldots, \ell\}$, $S_2 = \{\ell + r, \ell + 2r, . . . , \ell + (n-2)r\}$, and $S=S_1 \cup S_2$.
	It is clear that if $r=1$, then $S$ is a complete system of representatives modulo $a$. 
	Suppose that $1<r\le n-2$ and we will look for a contradiction.

 	First, suppose $1<r\le\ell$. 
 	Then we can write $n-1=pr+q$ for some integers $p \ge 0$ and $0\le q<r$. 
 	Since $p = \frac{n-1-q}{r} \le \frac{n-1}{2}$ and $n \ge 4$, we have $1\le p\le n-3$. 
	We then have $$a=\ell+n-1=\ell+pr+q=\ell+(p+1)r - (r-q).$$
	Hence, $\ell+(p+1)r\equiv r-q \pmod a$. 
	Observe that $0<r-q \le r\le \ell$. 
	This implies that $\ell+(p+1)r \in S_2$ and $r-q \in S_1$ are distinct numbers in $S$ that are congruent modulo $a$. 
	Hence, $S$ is not a complete system of representatives.
	Therefore, we must have $\ell<r\le n-2$. 

	Since $S$ is assumed to be a complete system of representatives modulo $a$, there exists $s \in S$ such that $r \equiv s \pmod a$.
	If $s\in S_1$, then $r-s\ge r-\ell>0$. On the other hand, since $a>n-2\ge r\ge r-s$, we have $0<r-s<a$ which contradicts the assumption that  $r\equiv s\pmod a$. 
	Therefore, $s \in S_2$, and we can write $s=\ell+kr$ for some $1\le k\le n-2$. 
	Then, $s-r=\ell+(k-1)r$. However, this is impossible because $0$ and $\ell+(k-1)r$ are distinct representatives in $S$, which contradicts the assumption that $S$ is a complete system of representatives modulo $a$.
\end{proof}

	We begin with the following.
	
	\begin{lem}\label{AperySec3}
		Let $H = \left<a_1, a_2,\ldots , a_n\right>$ be a numerical semigroup with embedding dimension $n$ and multiplicity $a_1$.
		Suppose $k[H]/(t^{a_1})$ is stretched and $\PF(H) =\{h+\alpha, h+2\alpha, \ldots , h+(n-1)\alpha\}$ for some integers $h \ge 0$ and $\alpha > 0$.
		Put $\ell = a_1 - n +1$.
		Then
		$$
		\Apery(H,a_1) = \{0, a_2, a_3,\ldots , a_n\} \cup \{2 a_\lambda, 3a_\lambda,\ldots , \ell a_\lambda\}
		$$
		for some $2 \le \lambda \le n$.
	\end{lem}

	\begin{proof}
		If $\ell =1$ or $\ell \ge 3$, then the result follows directly from Lemma \ref{Apery}. Thus, we may assume $\ell = 2$.

		Since $\PF(H) = \{h+\alpha, h+2\alpha, \ldots , h+(n-1) \alpha\}$, we have
		$$\Apery(H, a_1) = \{0, h+\alpha + a_1, h+2\alpha + a_1, \ldots , h+(n-1)\alpha + a_1\} \cup \{\beta + a_1\}$$ 
		for some $\beta \in \mathbb{N} \setminus H$.
		If $\{a_2, a_3,\ldots , a_n\} = \{h + i\alpha + a_1 \mid 1 \le i \le n-1\}$, then it follows that $\PF(H) = \{a_2-a_1, \ldots, a_n-a_1\}$.
		This leads to a contradiction as $H$ does not have maximal embedding dimension.
		Therefore, there exists $2 \le \lambda \le n$ such that $a_\lambda = \beta + a_1$.

		On the other hand, since $\beta \notin H \cup \PF(H)$, there exists $\gamma \in \PF(H)$ such that $0< \gamma - \beta \in H$. 
		We may write $\gamma = h + j \alpha$ for some $1 \le j \le n-1$.
		Then, we have
		$$\gamma - \beta = (h+ j\alpha) - (a_\lambda - a_1) = c_1 a_1 + \cdots + c_n a_n$$ 
		for some $c_1, \ldots , c_n \ge 0$ with $c_1 + \cdots + c_n > 0$.

		If $c_1 > 0$, then $h + j \alpha - a_\lambda \in H$, which is a contradiction because $h + j\alpha \notin H$. 
		Thus $c_1 = 0$.
		Next, suppose there exists an index $2 \le i \le n$ such that $i \ne \lambda$ and $c_i > 0$.
		Then, $(h+ j\alpha + a_1) -a_i \in H$. The maximality of $a_i$ in $\Apery(H, a_1)$ implies that $a_i = h+ j \alpha + a_1$.
		Hence, we have
		$$
			(1-c_i)a_i = \sum_{2 \le k \le n,~k \notin \{i,\lambda\}} c_k a_k + (c_\lambda+1) a_\lambda.
		$$
		However, this is impossible because the left-hand side is non-positive, while the right-hand side is positive.
		Therefore $c_i = 0$ for each $i\in \{1,2, \ldots , n\} \setminus\{\lambda\}$ and whence $c_\lambda \ge 1$.

		Now, suppose $c_\lambda \ge 2$.
		Since $(c_\lambda + 1) a_\lambda = h + j \alpha + a_1 \in \Apery(H, a_1)$, it follows that $2a_\lambda- a_1, 3a_\lambda- a_1 \notin H$, otherwise we wolud have $h + j \alpha = (c_\lambda +1) a_\lambda - a_1 \in H$ which is a contradiction.
		However, this implies the impossible inequality $a_1 > n + 1 = n + \ell - 1$.
		Therefore we have $c_\lambda = 1$ so that $h + j \alpha + a_1 = 2a_\lambda$.
		This shows that $$\Apery(H,a_1) = \{0, a_2, a_3, \ldots , a_n\} \cup \{2 a_\lambda\}.$$
	\end{proof}

	\begin{prop}\label{j}
		Under the notation given in Lemma \ref{AperySec3}, assume that $\ell \ge 2$ and $n \ge 3$. We then have the following.
		\begin{enumerate}
			\item There exists $1 \le j \le n-1$ such that 
		$$\left\{\ell + n -1, \frac{h+ j\alpha + a}{\ell} \right\} \cup \{h + \alpha + a, h + 2 \alpha + a,\ldots , h+(n-1)\alpha + a,  \} \setminus \{h + j \alpha + a\}$$
		is the minimal system of generators of $H$. 
			\item The $j$ in the assertion (1) must be $1$ or $n-1$.
		\end{enumerate}
	\end{prop}

	\begin{proof}
		(1) It follows directly from Lemma \ref{AperySec3}, as $H=\left<\Apery(H,a_1) \cup \{a_1\}\right>$.
		
		(2) Suppose $2 \le j \le n-2$. Define $a_{j+1} = \frac{h+ j\alpha + a}{\ell}$ and $a_{i+1} = h + i\alpha + a$ for $1 \le i \le n-1$ with $i \ne j$.
		Then $2 \ell a_{j+1} = a_j + a_{j+2}$. 
		Now, assume $a_j + a_{j+2} \in \Apery(H, a_1) = \{0, a_2, \ldots , a_n\} \cup \{2a_j, 3a_j,\ldots , \ell a_j\}$.
		This implies $a_j + a_{j+2} = q a_j$ for some $2 \le q \le \ell$, which leads to a contradiction since $2\ell = q \le \ell$.
		Hence, $a_j + a_{j+2} \notin \Apery(H, a_1)$.
		Therefore, $X_{j+1}^{2\ell} \in I_H + (X_1)$.
		Notice that it is part of a system of generators of $I_H + (X_1)$, because $X_{j+1}^{2\ell} - X_j X_{j+2}$ is part of a system of generators of $I_H$, which follows from the fact that the total degree of $X_j X_{j+2}$ is $2$.
		However, since $\fkm^{\ell+1} = (0)$ where $\fkm$ denotes the maximal ideal of $k[H]/(t^{a_1})$, we have $X_{j+1}^{\ell+1} \in I_H + (X_1)$.
		Thus, $2\ell \le \ell + 1$, which is a contradiction because $\ell \ge 2$.
		Therefore, $j = 1$ or $j=n-1$.
	\end{proof}

	\begin{lem}\label{gcd}
		Under the notation as in Proposition \ref{j}, the greatest common divisor of $a = \ell + n - 1$ and $b = \frac{h+j\alpha + a}{\ell}$ must be $1$.
	\end{lem}

	\begin{proof}
		Define $a_1 = a$, $a_2 = b$, and $$a_i = 
		\begin{cases}
 			h + (i-1) \alpha + a & (j=2)\\
 			h + (i-2) \alpha + a & (j= n-1)
		\end{cases}$$
		for $3 \le i\le n$. Let $d = \gcd(a,b)$ be the greatest common divisor of $a$ and $b$.
		Suppose $d \ge 2$.
		Then, $$\sharp \{0 \le p \le a-1 \mid \exists v \in \mathbb{Z}, p \equiv vb \pmod a\} = \frac{a}{d}.$$
		On the other hand, since $\mathbb{Z}/a\mathbb{Z} = \{\overline{q} \mid q \in \Apery(H, a_1)\} = \{\overline{a_3}, \overline{a_4},\ldots , \overline{a_n}\} \cup \{\overline{0}, \overline{b}, \overline{2b},\ldots , \overline{\ell b}\}$, we have 
		$$\sharp\{3 \le i \le n \mid \exists v \in \mathbb{Z}, a_i \equiv vb \pmod a\} = \frac{a}{d} - (\ell + 1).$$
		If $\frac{a}{d} = \ell +1$, then there is no $3 \le i \le n$ such that $a_i \equiv vb \pmod a$ for some $v \in \mathbb{Z}$.
		Hence, we would have $n-2 < d$, which is impossible because 
		$$
		n-2 = a - (\ell + 1) = d(\ell +1) - (\ell +1) = d + ((d-1) \ell -1) > d.
		$$
		Therefore, $\frac{a}{d} > \ell + 1$ and there at least one $3 \le i \le n$ such that $a_i \equiv vb \pmod a$ for some $v \in \mathbb{Z}$.

		For $3 \le i \le n-d$, if $a_i \equiv vb \pmod a$ for some $v \in \mathbb{Z}$, then $$a_{i+d} = a_i + d\alpha = a_i + (ax + by) \alpha \equiv (v+y)b \pmod a$$ where $d = ax + by$ for some $x,y \in \mathbb{Z}$.
		Furthermore, the converse is also true. 
		Hence, for $3 \le i \le n-d$, $a_i \equiv vb \pmod a$ for some $v \in \mathbb{Z}$ if and only if $a_{i+d} \equiv ub \pmod a$ for some $u \in \mathbb{Z}$.
		From this observation, we have
		 $$\sharp\{3 \le i \le n \mid \exists v \in \mathbb{Z}, a_i \equiv vb \pmod a\} \ge  \left[\frac{n-2}{d}\right],$$
		where $[*]$ denotes the Gauss symbol (floor function).

		Thus, $\frac{a}{d} - (\ell +1) \ge \left[\frac{n-2}{d}\right] > \frac{n-2}{d} - 1$.
		However, this implies $(1-d)\ell + 1 > 0$, which is impossible because $d \ge 2$ and $\ell \ge 2$. Therefore, $d = 1$.
	\end{proof}

	Now, we can prove Theorem \ref{thmSec3} by combining the following Theorems \ref{thm:j=1} and \ref{thm:j=n-1}.

	\begin{thm}\label{thm:j=1}
		Under the notation as in Proposition \ref{j}, suppose $j=1$. Put $a = \ell + n -1$ and $b = \frac{h + \alpha + a}{\ell}$. Then we have the following.
		\begin{enumerate}
			\item $h - (\ell -1) b \equiv 0 \pmod a$ and there exists $0 \le h_1 \in \mathbb{Z}$ such that $h = h_1 a + (\ell -1) b$.
			\item $I_H = \detid_2 \begin{pmatrix}
 				X_1^{h_1+1} & X_2^{\ell} & X_3 & \cdots & X_{n-1} & X_n\\
 				X_2 & X_3 & X_4 & \cdots & X_n & X_1^{(h_1 + 1)\ell + \alpha}
 			\end{pmatrix}$, where $a_1 = a$, $a_2 = b$, and $a_i = h+(i-1)\alpha + a$ for $3 \le i \le n$.
		\end{enumerate}
	\end{thm}

	\begin{proof}
		(1) 	Since $\gcd(a,b) = 1$ by Lemma \ref{gcd}, $\{0, b, 2b,\ldots , \ell b\} \cup \{ \ell b + \alpha, \ell b  + 2 \alpha, \ldots , \ell b + (n-2)\alpha\}$ and $\{0, b, 2b, \ldots, (\ell+n-2)b\}$ are complete systems of representatives modulo $a$.
		Hence, we have
		$$
		\{\overline{(\ell+1)b}, \overline{(\ell+2)b}, \ldots , \overline{(\ell+n-2)b}\} = \{\overline{\ell b + \alpha}, \overline{\ell b + 2\alpha}, \ldots , \overline{\ell b + (n-2)\alpha}\}
		$$
		in $\mathbb{Z} / a\mathbb{Z}$.
		Therefore $\ell b + \alpha \equiv (\ell + r) b \pmod a$ for some $1 \le r \le n-2$, whence $\alpha \equiv rb \pmod a$.
		Thus, $\{0,b,2b,\ldots , \ell b\} \cup \{(\ell +r)b, (\ell+2r)b, \ldots , (\ell+(n-2)r)b\}$ is a complete system of representatives modulo $a$.
		By $\gcd(a,b) = 1$ again, we have $\{0,1,2,\ldots , \ell\}\cup\{\ell+r, \ell+2r,\ldots , \ell+(n-2)r\}$ is a complete system of representatives modulo $a$.
		Thus, thanks to Lemma \ref{elementary}, we conclude that $r=1$, and therefore $\alpha \equiv b \pmod a$.
		
		Now, since $\ell b = h + \alpha + a$, it follows that $h - (\ell-1) b = b - \alpha + a \equiv 0$.
		Hence $h = h_1 a + (\ell-1) b$ for some $h_1 \in \mathbb{Z}$. 
		Finally, note that $h \in H$ and $\ell-1 < a$. Thus, $h_1$ must be non-negative, and we get the desired conclusion.

		(2) Put $J=\detid_2
	\left(\begin{smallmatrix}
		X_1^{h_1+1} & X_2^\ell & X_3 & \cdots &X_{n-1} & X_n\\
		X_2 & X_3 & X_4 & \cdots & X_n & X_1^{(h_1+1)\ell + \alpha}
	\end{smallmatrix}\right)$. 
	Then, 
	\begin{eqnarray*}
		J+(X_1) &=& \detid_2
		\begin{pmatrix}
			0 & X_2^\ell & X_3 & \cdots &X_{n-1} & X_n\\
			X_2 & X_3 & X_4 & \cdots & X_n & 0
		\end{pmatrix} + (X_1)\\
		&=& (X_1, X_2^{\ell+1}) + \left(X_iX_j \mid 2\le i\le j \le n, (i,j) \ne (2,2)\right).
	\end{eqnarray*}
	Hence, $\dim_k S/(J+(X_1)) = \ell + n -1 = a_1$. Therefore, thanks to Proposition \ref{Nak} and Corollary \ref{commondiff}, it is enough to show that $a_2 - (h_1+1)a_1 = ((h_1+1)\ell + \alpha)a_1 - a_n = \alpha$.
	Since $b = \frac{1}{\ell}((h_1+1)a + (\ell-1)b + \alpha)$, we have $b = (h_1 + 1)a + \alpha$. Hence $a_2 - (h_1+1)a_1 = b-(h_1+1)a = \alpha$. 
	On the other hand, the direct computation shows that $((h_1+1) \ell + \alpha)a_1 - a_n = \alpha$. 
	Thus, we have $I_H = J$.
	\end{proof}

	\begin{thm}\label{thm:j=n-1}
		Under the notation as in Proposition \ref{j}, suppose $j=n-1$. Put $a = \ell + n -1$ and $b = \frac{h + (n-1)\alpha + a}{\ell}$. Then we have the following.
		\begin{enumerate}
			\item $b + \alpha \equiv 0 \pmod a$.
			\item $h \equiv 0 \pmod a$.
			\item Put $h_1 = \frac{h}{a}$. Then $\frac{b+\alpha}{a} = \frac{h_1 + 1 + \alpha}{\ell}$.
			\item $I_H = \detid_2 \begin{pmatrix}
 				X_1^{h_1+1} & X_2 & \cdots & X_{n-1} & X_n\\
 				X_2 & X_3 & \cdots & X_n^{\ell} & X_1^{\frac{h_1+1+\alpha}{\ell}}
 			\end{pmatrix}$, where $a_1 = a$, $a_n = b$, and $a_i = h+(i-1)\alpha + a$ for $2 \le i \le n-1$.
		\end{enumerate}
	\end{thm}

	\begin{proof}
		(1)	Since $\gcd(a,b) = 1$ by Lemma \ref{gcd}, the sets $\{0, b, 2b,\ldots , \ell b\} \cup \{ \ell b - \alpha, \ell b  - 2 \alpha, \ldots , \ell b - (n-2)\alpha\}$ and $\{0, b, 2b, \ldots, (\ell+n-2)b\}$ are complete systems of representatives modulo $a$.
		Using similar considerations as in the proof of Theorem \ref{thm:j=1}, we have $b \equiv -\alpha \pmod a$.
		
		(2) + (3) Since $\ell b = h + (n-1)\alpha + a \equiv - \ell \alpha \pmod a$, we have $h + (n+\ell-1)\alpha + a \equiv h \equiv 0 \pmod a$ (Notice that $n+\ell -1 = a$). 
		Thus, we can write $h = h_1 a$ for some $h_1 \ge 0$.
		
		Now, observe that $\ell(b+\alpha) = (h_1 + 1)a + (\ell + n -1)\alpha = (h_1 + 1 + \alpha) a$. Therefore 
		$$
		\frac{b+\alpha}{a} = \frac{\ell(b+\alpha)}{\ell a} = \frac{h_1 + 1 + \alpha}{\ell}.
		$$
		
		(4) By a similar computation as in the proof of Theorem \ref{thm:j=1}, we obtain
		$$
		I_H = \detid_2
		\begin{pmatrix}
			X_1^{h_1+1} & X_2 & \cdots & X_{n-1} & X_n\\
			X_2 & X_3 & \cdots & X_n^\ell & X_1^{\frac{h_1+1+\alpha}{\ell}}	
		\end{pmatrix}.
		$$
	\end{proof}

\section{Cohen-Macaulayness of $\gr(k[H])$}

In this section, we explore the Cohen-Macaulay property of the tangent cone $\gr(k[H]) = \bigoplus_{n \ge 0} \fkm^n/\fkm^{n+1}$ of the numerical semigroup ring $k[H]$, when $k[H]/(t^{a_1})$ is stretched, where $\fkm$ denotes the graded maximal ideal of $k[H]$.

Throughout this section, let $H$ be a numerical semigroup satisfying the condition in Theorem \ref{thm:j=1} or Theorem \ref{thm:j=n-1}.
In each case, $\fkm^{\ell+1} \subseteq (t^{a_1})$ and $\fkm^\ell  \not\subseteq (t^{a_1})$. 
Therefore, it is known by \cite[Corollary 2.4]{sally1} that $\gr(k[H])$ is Cohen-Macaulay if and only if $t^{(\ell + 1) b} \in t^{a_1} \fkm^{\ell}$.
Notice that, $k[H] \cong k[X_1, X_2,\ldots , X_n]/I_H$ where $I_H$ is the same form as in Theorem \ref{thm:j=1} (resp. Theorem \ref{thm:j=n-1}) when $j=1$ (resp. $j=n-1$).

\begin{thm}\label{gr:j=1}
	Under the notation as in Theorem \ref{thm:j=1}, $\gr(k[H])$ is Cohen-Macaulay if and only if $h_1 +1 \ge \ell$.
\end{thm}

\begin{proof}
	By considering the form of the defining ideal given in Theorem \ref{thm:j=1}, we immediately get that $(\ell+1) b = (h_1+1) a_1 + a_3$.
	Furthermore, we know that there are no other expressions of $(\ell+1)b$ in $H$ by the form of the defining ideal again.
	Therefore, $t^{(\ell + 1)b -a} = t^{h_1 a_1 + a_3} \in \fkm^\ell$ if and only if $h_1+1 \ge \ell$.
\end{proof}

\begin{thm}\label{gr:j=n-1}
	Under the notation as in Theorem \ref{thm:j=n-1}, we have the following.
	\begin{enumerate}
		\item 	$\gr(k[H])$ is Cohen-Macaulay if and only if $\frac{h_1 + 1 + \alpha}{\ell} \ge \ell$.
		\item If $h_1 \ge \ell^2 - \ell - \alpha$, then $\gr(k[H])$ is Cohen-Macaulay. 
		\item If $\ell = 2$, then $\gr(k[H])$ is Cohen-Macaulay.
	\end{enumerate}
\end{thm}

\begin{proof}
	(1) By Theorem \ref{thm:j=n-1}, we have $(\ell+1) b = \frac{h_1+1+\alpha}{\ell} a_1 + a_{n-1}$.
	We may even find other expressions by using the relation $X_1^{h_1+1} X_{n-1} \equiv X_2 X_{n-2} \mod I_H$, the total degree of the left-hand side is greater than or equal to that of the right-hand side.
	Hence, $t^{(\ell+1) b - a} \in \fkm^{\ell}$ if and only if $\frac{h_1 + 1 + \alpha}{\ell} \ge \ell$.
	
	(2) Suppose $h_1 \ge \ell^2 -\ell -\alpha$. Then $h_1 + 1 + \alpha \ge \ell(\ell -1) + 1$. Since $\frac{h_1 + 1 + \alpha}{\ell} \in \mathbb{Z}$, it is at least $\ell$. Therefore, in this case, $\gr(k[[H]])$ is Cohen-Macaulay thanks to (1).
	
	(3) If $h_1 \ge 1$, then $h_1 + 1 + \alpha \ge 3$. Hence, since $\ell = 2$ and $\frac{h_1 +1 + \alpha}{\ell} \in \mathbb{Z}$, we have $\frac{h_1 + 1 + \alpha}{\ell} \ge 2 = \ell$. By (1), we conclude that $\gr(k[H])$ is Cohen-Macaulay.
	
	On the other hand, suppose $h_1 = 0$.
	If $\alpha = 1$, then $a_n = \frac{a + (n-1)}{2} < a$. 
	However, it is impossible because we assume that $a$ is the multiplicity of $H$.
	Hence $\alpha \ge 2$. Then $h_1 + 1 + \alpha = 1 + \alpha \ge 3$ and this is even. Hence $\frac{h_1 + 1 + \alpha}{\ell} \ge 2 = \ell$. Hence $\gr(k[H])$ is Cohen-Macaulay again by (1).
\end{proof}

\begin{cor}
	Under the notation as in Proposition \ref{j}, we further assume that $H$ has almost maximal embedding dimension, that is, $\ell = 2$. Then $\gr(k[H])$ is NOT Cohen-Macaulay if and only if $j=1$ and $h = b = a+\alpha$.
\end{cor}

\section{Examples}

In the last section of this paper, we guarantee that examples satisfying the assumptions of our results exist by providing concrete examples.
In this section, $k$ denotes an arbitrary field the same as in previous sections.

\begin{ex}\label{ex1}
	Let $ \ell \geq 2 $, $ n \geq 3 $, $ \alpha > 0 $, and $ h_1 \geq 0 $ be integers. Put $ a = \ell + n - 1 $ and assume that $ \gcd(a, \alpha) > 1 $. 
	\begin{enumerate}
		\item Let $ b = (h_1 + 1)a + \alpha $. Then, the numerical semigroup 
		$$
		H = \langle a, b, \ell b + \alpha, \ell b + 2\alpha, \ldots, \ell b + (n - 2)\alpha \rangle
		$$
		satisfies the assumptions of Theorem \ref{thm:j=1}.
		
		\item Suppose that $ \frac{h_1 + 1 + \alpha}{\ell} \in \mathbb{Z} $ and $ \frac{h_1 + 1 + \alpha}{\ell} a - \alpha > a $. Let $ b = \frac{h_1 + 1 + \alpha}{\ell} a - \alpha $. Then, the numerical semigroup 
		$$
		H = \langle a, \ell b - (n - 2)\alpha, \ell b - (n - 3)\alpha, \ldots, \ell b - \alpha, b \rangle
		$$
		satisfies the assumptions of Theorem \ref{thm:j=n-1}.
	\end{enumerate}
\end{ex}


\begin{ex}
Let $H=\left<a_1 = 6, a_2 = 13, a_3=40, a_4 = 41\right>$, which is the case (1) of Example \ref{ex1} where $\ell=3$, $n=4$, $\alpha = 1$, and $h_1 = 1$. 
	Then we have the following.

	\begin{enumerate}
		\item $\Apery(H,6)=\{0,13,40,41\}\cup\{26,39\}$. Hence $k[H]/(t^6)$ is stretched.
		\item  $\PF(H) = \{33,34,35\}$.
		\item $I_H = \detid_2 \begin{pmatrix}
	X_1^2 & X_2^3 & X_3 & X_4\\
	X_2 & X_3 & X_4 & X_1^7
\end{pmatrix}$.
		\item Since $h_1 + 1 = 1 < 3 = \ell$, $\gr(k[H])$ is not Cohen-Macaulay by Theorem \ref{gr:j=1}.
	\end{enumerate}
\end{ex}

\begin{ex}
	Let $H=\left<a_1 = 7, a_2 = 39, a_3 = 43, a_4 = 47, a_5 = 17\right>$, which is the case (2) of Example \ref{ex1} where $\ell =3$, $n=5$, $\alpha = 4$, and $h_1 = 4$. 
	Then we have the following.

	\begin{enumerate}
		\item $\Apery(H, 7) = \{0,39,43,47,17\}\cup\{34,51\}$. Hence $k[H]/(t^7)$ is stretched.
		\item $\PF(H) = \{32,36,40,44\}$.
		\item $I_H=\detid_2 \begin{pmatrix}
	X_1^5 & X_2 & X_3 & X_4 & X_5\\
	X_2 & X_3 & X_4 & X_5^3 & X_1^3
\end{pmatrix}$.
		\item Since $h_1 = 4 > 2 = \ell^2 - \ell - \alpha$, $\gr(k[H])$ is Cohen-Macaulay by Theorem \ref{gr:j=n-1}.
	\end{enumerate}

\end{ex}

The following example says that the condition $\sharp \PF(H) = n-1$ can not be removed from our argument.

\begin{ex}[\cite{goto}]\label{example2}
Let $H=\left<a_1 = 6, a_2 = 11, a_3 = 13, a_4 = 16, a_5 = 20\right>$. Then $\Apery(H,6)=\{0,11,13,16,20\}\cup\{27\}$. This also implies by Lemma \ref{Apery} that $k[H]/(t^6)$ is stretched. We have $\PF(H) = \{7,14,21\}$, whence $H$ does not satisfy the condition (3) in Theorem \ref{main theorem} because it is an arithmetic sequence of length $3 \ne 5-1$. Therefore the defining ideal $I_H$ of $k[H]$ can not have the form as in the condition (2). In fact, a direct computation gives 
$$
I_H=\rmI_2 \begin{pmatrix}
	X_3 & X_1^3 & X_5 & X_1^2X_2 & X_2X_4\\
	X_1 & X_2 & X_3 & X_4 & X_5
\end{pmatrix} + (X_2^2-X_1X_4, X_4^2-X_1^2X_5).
$$
\end{ex}

Remember that the assumption that $I_H$ is generated by $2 \times 2$ minors of a $2 \times n$ matrix implies $\sharp \PF(H) = n-1$.

The next example shows how the condition that $k[H]/(t^{a_1})$ is stretched gives a strong restriction on the defining ideals $I_H$, and also demonstrates that $I_H$ certainly does not have the desired system of generators when $\PF(H)$ does not form an arithmetic sequence, even if $\sharp \PF(H) = n-1$.

\begin{ex}
	Let $H=\left<a_1 = 8,a_2 = 9,a_3 = 31,a_4 = 37,a_5 = 38\right>$. Then we have the following.
	
	\begin{enumerate}
		\item $\Apery(H,8) = \{0,9,31,37,38\}\cup\{18,27,36\}$. Hence $k[H]/(t^8)$ is stretched.
		\item $\PF(H) = \{23,28,29,30\}$.
		\item $k[H]/(t^{a_1}) \cong k[X_2,X_3,X_4,X_5]/J$, where
		$$
		J= (X_2^5, X_2X_3, X_2X_4, X_2X_5, X_3^2, X_3X_4, X_3X_5, X_4^2,X_4X_5, X_5^2).
		$$
		\item $I_H = (X_2^5 - X_1X_4, X_2 X_3-X_1^5, X_2X_4 - X_1X_5, X_2X_5 - X_1^2X_3, X_3^2 - X_1X_2^6, X_3 X_4 - X_1^4X_2^4, X_3X_5 - X_1^3X_2^5, X_4^2 - X_1^2X_2^3X_3, X_4X_5 - X_1X_2^4X_3, X_5^2-X_1X_3X_4)$. Furthermore, it can be expressed as follows by using the matrix representation:
		$$
		I_H = \detid_2 
		\begin{pmatrix}
			X_1 & X_2^4 & X_3 & X_4 & X_5\\
			X_2 & X_4 & X_1^4 & X_5 & X_1X_3
		\end{pmatrix} + \detid_2
		\begin{pmatrix}
			X_1^3 & X_2 & X_3 & X_4 & X_5\\
			X_3 & X_1^2 & X_5 & X_1X_2^4 & X_2^5
		\end{pmatrix} 
		$$
	\end{enumerate}
\end{ex}

We close with the following example, which leads us to expect that Conjecture \ref{conj} has an affirmative answer for broader classes of numerical semigroups.

\begin{ex}
	Let $H=\left<a_1 = 8, a_2 = 9 ,a_3 = 28,a_4 = 29,a_5 = 15\right>$.
	Then we have the following.
	
	\begin{enumerate}
		\item $\Apery(H, 8) = \{0,9,28,29,15\} \cup \{18,27,30\}$. Hence $k[H]/(t^8)$ is not stretched.
		\item $\PF(H) = \{10,11,12,13\}$.
		\item $I_H=\detid_2\begin{pmatrix}X_1 & X_2^3 & X_3 & X_4 & X_5\\
			X_2 & X_3 & X_4 & X_5^2 & X_1^2
			\end{pmatrix}.
			$
	\end{enumerate}
\end{ex}

\section*{Ackknowledgements}

The authors are grateful to P. L. Huong, S. Iai, and M. Ikuma for their helpful discussions.

\end{document}